\documentclass[a4paper,11pt]{article}
\usepackage{amsmath}
\usepackage{amssymb}
\usepackage{amsthm}
\usepackage[all]{xy}
\usepackage[latin1]{inputenc}        
\usepackage[dvips]{graphics}
\usepackage[dvips]{graphicx}
\usepackage{amscd}
\usepackage{mathrsfs}
\usepackage[mathcal]{eucal}
\usepackage{fullpage}
\usepackage{caption}
\usepackage{float}

\author{Cinzia Bisi\thanks{Partially supported by Progetto MIUR di
Rilevante Interesse Nazionale {\it Propriet{\`a} geometriche
delle variet{\`a} reali e complesse} and by GNSAGA - INDAM. \smallskip \newline
} \,, Francesco Polizzi}
\title{On proper polynomial maps of $\mC^2$}

\date{}

\newtheorem{inizio}{Lemma}[section]
\newtheorem{theorem}[inizio]{Theorem}
\newtheorem{corollary}[inizio]{Corollary}
\newtheorem{proposition}[inizio]{Proposition}
\newtheorem{remark}[inizio]{Remark}
\newtheorem{lemma}[inizio]{Lemma}
\newtheorem{claim}[inizio]{Claim}
\newtheorem{definition}[inizio]{Definition}
\newtheorem{question}[inizio]{Question}
\newtheorem{o-problem}[inizio]{Open Problem}

\newtheorem*{teo-L}{Theorem}
\newtheorem*{teoA}{Theorem A}
\newtheorem*{teoB}{Theorem B}
\newtheorem*{teoC}{Theorem C}

\newtheorem*{corollary-s}{Corollary}
\theoremstyle{definition}
\newtheorem{example}[inizio]{Example}

\newcommand{\lr}{\longrightarrow}

\newcommand{\mC}{\mathbb{C}}

\begin{document}


\maketitle


\abstract Two proper polynomial maps $f_1, \,f_2 \colon \mC^2 \lr
\mC^2$ are said to be \emph{equivalent} if there exist $\Phi_1,\,
\Phi_2 \in \textrm{Aut}(\mC^2)$ such that $f_2=\Phi_2 \circ f_1
\circ \Phi_1$. We investigate proper polynomial maps of
topological degree $d \geq 2$ up to equivalence. Under the further
assumption that the maps are Galois coverings we also provide the
complete description of equivalence classes. This widely extends
previous results obtained by Lamy in the case $d=2$.
\endabstract

\section{Introduction} \label{sec:intro}
Let $f \colon \mathbb{C}^2 \lr \mathbb{C}^2$ be a dominant polynomial map.
We say that $f$ is {\it proper} if it is closed and for every point $y \in \mathbb{C}^2$
the set $f^{-1}(y)$ is compact. The {\it topological degree} $d$ of $f$
is defined as the number of preimages of a general point.\\
The semi-group of proper polynomial maps from $\mathbb{C}^2$ to $\mathbb{C}^2$
is not completely understood yet.
It is known that these maps cannot provide any counterexample
to the Jacobian Conjecture, see \cite[Theorem 2.1]{BCW82}.
Nevertheless, it is worthwhile to study them from other points of view,
for instance analyzing their dynamical behaviour;
this investigation was recently started in \cite{DS03}, \cite{DS08}, \cite{FJ07} and \cite{FJArx07}.
In the present paper we do not consider any dynamical question but we
 try to generalize to arbitrary $d \geq 3$
the following theorem, proved in \cite{Lam05}.
\begin{theorem}[Lamy] \label{teo:Lamy}
Let $f \colon \mC^2 \lr \mC^2$ be a proper polynomial map of
topological degree $2$. Then there exist $\Phi_1,\, \Phi_2 \in
\emph{Aut}(\mC^2)$ such that
\begin{equation*}
f=\Phi_2 \circ \tilde{f} \circ \Phi_1,
\end{equation*}
where $\tilde{f}(x, \, y)=(x, y^2)$.
\end{theorem}
We say that two proper polynomial maps $f_1, \,f_2 \colon \mC^2 \lr
\mC^2$ are \emph{equivalent} if there exist $\Phi_1,\, \Phi_2 \in
\textrm{Aut}(\mC^2)$ such that
\begin{equation*}
f_2=\Phi_2 \circ f_1  \circ \Phi_1.
\end{equation*}
One immediately check that equivalent maps have the same topological degree.
Therefore Theorem \ref{teo:Lamy} says that when $d=2$ there is just one equivalence
class, namely that of $\tilde{f}$. \\
The aim of our work is to answer some questions that naturally arise from
Lamy's result. The first one, already stated in \cite{Lam05}, is the following:
\begin{question}
Is every proper polynomial map $ f \colon \mC^2 \lr \mC^2$
equivalent to some map of type $(x, \, y) \lr (x, \, P(y))?$
\end{question}
The answer is negative, and a counterexample
 is provided already in degree $3$ by the proper
 map $f \colon \mathbb{C}^2 \lr \mathbb{C}^2$
 given by
 \begin{equation*}
f(x, \, y)=(x, \, y^3+xy).
\end{equation*}
This map was first considered by Whitney; clearly it is not equivalent to
 any map of the form $(x, \, y) \lr (x, \, P(y))$, since its branch
 locus is the cuspidal cubic of equation $4x^3+27y^2=0$
 (see Remark \ref{rem:Witney-no-sep}). The particular form
 of this counterexample led us to the following very natural question:
\begin{question} \label{q-A}
Is every polynomial map $f \colon \mC^2 \lr \mC^2$ equivalent to
some map of type $(x, \, y) \lr (x, \, Q(x,y))?$
\end{question}
And, more generally:
\begin{question} \label{q-B}
How many equivalence classes of proper polynomial maps of fixed
topological degree $d \geq 3$ are there?
\end{question}
Answers to Questions \ref{q-A} and \ref{q-B} are the relevant results of
Section \ref{sec:proofs}, referred as Theorems A and B.
\begin{teoA}
For every $d \geq 3$ there exists at least one proper polynomial map
$f \colon \mC^2 \lr \mC^2$ such that $f$ is not equivalent to
 a map of type $(x, \, y) \lr (x, \, Q(x,y))$.
\end{teoA}
\begin{teoB}
For all positive integers $d, \, n$, with $d \geq 3$ and $n \geq 2$,
consider the polynomial map $f_{d, \,n} \colon \mathbb{C}^2 \lr
\mathbb{C}^2$ given by
\begin{equation*}
f_{d, \, n}(x,\,y):=(x, \, y^d-dx^ny).
\end{equation*}
Then $f_{d, \,n}$ and $f_{d, \, m}$ are equivalent if and only if
$m=n$. It follows that there exist infinitely many different
equivalence classes of proper polynomial maps $f \colon \mC^2 \lr
\mC^2$ of fixed topological degree $d$.
\end{teoB}
Comparing Theorem \ref{teo:Lamy} with Theorems A and B, one sees that the behaviour
of proper polynomial maps $\mathbb{C}^2 \lr \mathbb{C}^2$ up to equivalence
is completely different for $d=2$ and for $d \geq 3$. It seems that
a satisfactory description of all equivalence classes in the case
$d \geq 3$ is at the moment
out of reach; nevertheless, one could hope at least to classify those proper maps
enjoying some additional property. For this reason, in Section 3 we restrict our attention
 to polynomial maps $f \colon \mathbb{C}^2 \lr
\mathbb{C}^2$ which are \emph{Galois
coverings} with finite Galois group $G$.
All these maps are proper and their topological
degree equals $|G|$; moreover $G \subset \textrm{Aut}(\mathbb{C}^2)$ and
$f$ can be identified with the quotient map $\mathbb{C}^2 \lr \mathbb{C}^2/G$.
Since $G$ is a finite group, we may assume that
$G \subset \textrm{GL}(2,\, \mathbb{C})$ by a polynomial
change of coordinates, see \cite{Ka79}, and since $\mathbb{C}^2/G \cong \mathbb{C}^2$
it follows that $G$ is a {\it finite complex reflection group}.
These groups and their conjugacy classes in $\textrm{GL}(2, \, \mathbb{C})$
 were completely classified in \cite{ST54} and \cite{Coh76}. Therefore we may exploit
 this classification in order to prove the main result of Section \ref{sec:Galois}.
\begin{teoC}[see Theorem \ref{teo:Galois}]
Let $f \colon \mathbb{C}^2 \lr \mathbb{C}^2$ be a polynomial map
which is a Galois covering with finite Galois group $G$. Then, up to
equivalence, we are in one of the cases in Table
\emph{\ref{table:Galois}} of Section \emph{\ref{sec:Galois}}.
\end{teoC}
Referring to Table \ref{table:Galois}, we observe that the case $d=2$ corresponds to the map
$\mathfrak{f}_2$ (and to the Galois group $\mathbb{Z}_2$); therefore our Theorem C
widely extends Theorem \ref{teo:Lamy}. \\
We finally remark that the equivalence relation studied in
the present paper is weaker than the
conjugacy relation, in which we require $\Phi_2=\Phi_1^{-1}$. For instance, the two maps
$f_1(x, \, y)=(x, \, y^2)$ and $f_2(x, \, y)=(x, \, y^2+x)$ are equivalent
in our sense but they are not conjugate by any automorphism of $\mathbb{C}^2$,
since their sets of fixed points are not biholomorphic. The study of conjugacy classes
of proper maps of given topological degree is certainly an interesting problem,
but we will not consider it here: some good references are \cite{FJ07} and \cite{FJArx07}. \\
Some of our computations were carried out by using the Computer Algebra Systems
 \verb|GAP4| and \verb|Singular|, see \cite{GAP4} and \cite{SING}.
For the reader's convenience, we included the scripts in the Appendix.

\bigskip
\noindent $\mathbf{Acknowledgements.}$
The first author wishes to thank
Domenico Fiorenza for some fruitful conversations on the subject.

\section{Preliminaries} \label{sec:prel}

\subsection{Proper polynomial maps} \label{sub:prop}
We recall the following
\begin{definition}
Let $f \colon \mathbb{C}^n \lr \mathbb{C}^n$ be a dominant
polynomial map. We say that $f$ is \emph{proper} if it is closed and
for every point $p \in \mathbb{C}^n$ the set $f^{-1}(p)$ is compact.
Equivalently, $f$ is proper if and only if
for every compact set $K \subset \mathbb{C}^n$ the set $f^{-1}(K)$ is compact.
\end{definition}
Notice that, in the first part of the definition, the hypothesis $f$ closed is necessary.
For example, if one considers the map
$f(x, \, y)=(x+x^2y, \, y)$ on $\mathbb{C}^2,$ $f^{-1}(p)$ is compact  because it always consists of one or two points.
However: \\
$-$ $f$ is not closed, since the image of the curve $xy+1=0$ is the set of points $\{ (0,y) \,\,\, | \,\,\, y \in \mathbb{C}^* \};$ \\
$-$ $f$ is not proper, since for any compact
neighborhood $K$ of $(0, \, 0)$ the set $f^{-1}(K)$ is never compact, see \cite{Lam05}. \\
The map also provides an example of a  
surjective map which is not necessarily proper. On the other hand, every proper map must be surjective.\\
There is a purely algebraic condition for a polynomial map to be proper, see
\cite[Proposition 3]{Jel93}:
\begin{proposition} \label{prop:Jel}
A dominant polynomial map  $f \colon \mathbb{C}^n \lr \mathbb{C}^n$
is proper if and only if the push-forward map $f_{\ast} \colon
\mathbb{C}[s_1, \ldots, s_n] \lr \mathbb{C}[x_1, \ldots, x_n]$ is
finite, i.e., $f_{\ast}\mathbb{C}[s_1, \ldots, s_n] \subset
\mathbb{C}[x_1, \ldots, x_n]$ is an integral extension of rings.
\end{proposition}
In the sequel we will focus on proper polynomial maps $f \colon
\mathbb{C}^2 \lr \mathbb{C}^2$. We write
\begin{equation*}
f(x, \,y)= (f_1(x, \, y), \, f_2(x, \, y))
\end{equation*}
with $f_1, \, f_2 \in \mathbb{C}[x,y]$. Then the push-forward map
will be given by
\begin{equation*}
\begin{split}
f_{\ast} \colon \mathbb{C}[s, \, t] & \lr \mathbb{C}[x, \, y]\\
 & s \lr f_1(x, \, y) \\
 & t \lr f_2(x, \, y).
\end{split}
\end{equation*}
 Given such a map $f$, its \emph{Jacobian} $J_f$ is the polynomial
\begin{equation*}
J_f(x, \, y)= \left|
\begin{array}{cc}
\partial f_1/ \partial x  & \partial f_1/ \partial y \\
\partial f_2/ \partial x & \partial f_2/ \partial y \\
\end{array}
\right|.
\end{equation*}
The \emph{critical locus} $\textrm{Crit}(f)$ of $f$ is
 the affine variety  $V(J_f) \subset \mathbb{C}^2$.
 The \emph{branch locus} $B(f)$ of $f$ is the image of the critical
locus, that is $B(f)=f(\textrm{Crit}(f))$. Since $f$ is proper, the
restriction
\begin{equation*}
f \colon \mathbb{C}^2 \setminus f^{-1}(B(f)) \lr \mathbb{C}^2
\setminus B(f)
\end{equation*}
is an unramified covering of finite degree $d$; we will call $d$
the \emph{topological degree} of $f$.

\begin{definition} \label{def:equiv}
We say that two proper polynomial maps $f_1, \,f_2 \colon \mC^2
\lr \mC^2$ are \emph{equivalent} if there exist $\Phi_1,\, \Phi_2
\in \emph{Aut}(\mC^2)$ such that
\begin{equation} \label{eq:equiv}
f_2=\Phi_2 \circ f_1  \circ \Phi_1.
\end{equation}
\end{definition}
\begin{remark}
This equivalence relation in the semi-group of proper polynomial maps is weaker than the
conjugacy relation, in which we require $\Phi_2=\Phi_1^{-1}$. For instance, the two maps
$f_1(x, \, y)=(x, \, y^2)$ and $f_2(x, \, y)=(x, \, y^2+x)$ are equivalent
in our sense but they are not conjugate by any automorphism of $\mathbb{C}^2$,
since their sets of fixed points are not biholomorphic. The study of conjugacy classes
of proper polynomial maps of given topological degree is an interesting problem, but
we will not consider it in this paper.
\end{remark}
\begin{proposition} \label{crit-bihol}
If $f_1$ and $f_2$ are equivalent then they have the same
topological degree. Moreover $\emph{Crit}(f_1)$ is biholomorphic
to $\emph{Crit}(f_2)$ and $B(f_1)$ is biholomorphic to $B(f_2)$.
\end{proposition}
\begin{proof}
Assume that \eqref{eq:equiv} holds. Since $\Phi_1$ and $\Phi_2$
have topological degree $1$, it follows that $f_1$ and $f_2$ have
the same topological degree. By the chain rule we have
\begin{equation*}
J_{f_2}=J_{\Phi_2} \cdot J_{f_1} \cdot J_{\Phi_1},
\end{equation*}
so we obtain
\begin{equation*}
\textrm{Crit}(f_2)=\Phi_1^{-1}(\textrm{Crit}(f_1)), \quad
B(f_2)=\Phi_2(B(f_1))
\end{equation*}
and this completes the proof.
\end{proof}

\begin{definition}
We say that a polynomial map $f \colon \mathbb{C}^2 \lr
\mathbb{C}^2$ is \emph{semi-separate} if it is of the form
\begin{equation*}
f(x, \, y)=(x, \, Q(x, y)),
\end{equation*}
where $Q(x,y) \in \mathbb{C}[x, \,y]$. In particular we say that it
is \emph{separate} if it is of the form
\begin{equation*}
f(x, \, y)=(x, \, P(y)),
\end{equation*}
where $P(y) \in \mathbb{C}[y]$.
\end{definition}
Recall that a polynomial $Q(x,\,y) \in \mathbb{C}[x, \,y]$ is
called \emph{monic with respect to} $y$ if
\begin{equation*}
Q(x, \,y)=ay^n+ \textrm{terms of lower degree in }y, \quad a \in
\mathbb{C}^*.
\end{equation*}
By Proposition \ref{prop:Jel} it follows that a semi-separate
polynomial map $f$ is proper if and only if $Q(x, \, y)$ is monic
with respect to $y$; in this case, up to a dilation we may assume
that $f$ has the form
\begin{equation} \label{eq:semi-sep}
f(x, \, y)=(x, \, y^d+ q_{d-1}(x)y^{d-1}+ \cdots + q_0(x)),
\end{equation}
where $d$ is the topological degree. Notice that the Jacobian of
\eqref{eq:semi-sep} is
\begin{equation} \label{eq:crit-semisep}
J_f(x, \, y)=dy^{d-1}+(d-1)q_{d-1}(x)y^{d-2}+ \cdots + q_1(x).
\end{equation}
For example, let us consider the case of a general semi-separate
map $f \colon \mathbb{C}^2 \lr \mathbb{C}^2$ of topological degree
$3$. By using a linear transformation we can get rid of
 the term in $y^2$; therefore, up to equivalence, $f$ has the
form
\begin{equation*} \label{eq:d=3}
f(x, \, y)=(x, \, y^3+p(x)y+q(x)).
\end{equation*}
Then $\textrm{Crit}(f)$ has equation
\begin{equation*}
3y^2+p(x)=0,
\end{equation*}
whereas $B(f)$ has equation
\begin{equation*}
y^2-2q(x)y+ \frac{\Delta(x)}{27}=0,
\end{equation*}
where $\Delta(x):=27q(x)^2+4p(x)^3$ is the discriminant of
$y^3+p(x)y+q(x)$. In particular, taking $p(x)=x$ and $q(x)=0$, we
obtain the \emph{Whitney map}
\begin{equation*}
f(x, \, y)=(x, \, y^3+xy),
\end{equation*}
whose branch locus is the cuspidal cubic curve of equation
$4x^3+27y^2=0$.
\begin{remark} \label{rem:Witney-no-sep}
By \eqref{eq:crit-semisep} it follows that the branch locus of a separate map
is a disjoint union of lines. Therefore the previous computations together with
Proposition \emph{\ref{crit-bihol}} show that the Whitney map is not equivalent to
a separate one.
\end{remark}
The following lemma will be used in the proof of Theorem
A, see Section \ref{sec:proofs}.

\begin{lemma} \label{divides}
Let $f \colon \mathbb{C}^2 \lr \mathbb{C}^2$ be a semi-separate map
as in \eqref{eq:semi-sep}, with $d \geq 3$. If there exist two
polynomials $H_1(x, \, y)$, $H_2(x, \, y)$ such that
\begin{equation*}
J_f(x, \, y) = H_1(x, \, y)^{d-2}H_2(x,y),
\end{equation*}
then both affine curves $V(H_1)$ and $V(H_2)$ are biholomorphic to
$\mathbb{C}$.
\end{lemma}
\begin{proof}
By using \eqref{eq:crit-semisep}
 we can write
\begin{equation} \label{eq:crit-semisep-2}
dy^{d-1}+(d-1)q_{d-1}(x)y^{d-2}+ \cdots + q_1(x)=H_1(x,
\,y)^{d-2}H_2(x, \,y).
\end{equation}
The left-hand side of \eqref{eq:crit-semisep-2} is monic with
respect to $y$, so it cannot be divided by a polynomial in $x$. It
follows that both $H_1$ and $H_2$ contain $y$. Therefore, by
comparing the degrees, it follows that both $H_1$ and $H_2$ are
monic of degree $1$ in $y$, that is we may assume
\begin{equation*}
H_1(x, \, y)=y+h_1(x), \quad H_2(x, \, y)=dy+h_2(x),
\end{equation*}
for some $h_1(x), \, h_2(x) \in \mathbb{C}[x]$. This completes the
proof.
 \end{proof}

\subsection{Milnor number of a plane curve singularity}
In this subsection we summarize without proofs the definition and
the properties of the Milnor number of a plane curve singularity.
For further details we refer the reader to \cite[Chapter 1]{Lo84}
 and \cite[Chapter 3]{dJP00}. \\
Let $\mathbb{C} \{ x, \, y \}$ be the ring of convergent power
series in two variables; it is a local ring whose maximal ideal
$\mathfrak{m}$ consists of series with zero constant term, that is
of series vanishing at the point $o=(0, \,0)$.

\begin{definition}
A \emph{plane curve singularity} $X$ is a germ of an analytic
space $(V(F), \, o)$, where $F \in \mathfrak{m} \subset
\mathbb{C}\{x, \, y \}$.
\end{definition}

\begin{definition}
Let $X=(V(F), \, o)$ be a plane curve singularity. We define the
\emph{Milnor number } $\mu(X,\, o)$ by
\begin{equation*}
\mu(X, \, o):=\emph{dim}_{\mathbb{C}} \frac{\mathbb{C} \{ x, \, y\}
} {\big( \frac{\partial F}{\partial x}, \, \frac{\partial
F}{\partial y} \big)}.
\end{equation*}
\end{definition}

\begin{theorem} \label{milnor}
The Milnor number is well defined and
 it is an invariant of the singularity. Moreover
$\mu(X, \, o) < + \infty$
 if and only if $X$ is a germ of an \emph{isolated} plane curve
singularity.
\end{theorem}

\begin{example} \label{ex:milnor}
Set $F_{d,\,n}(x,\,y)= y^{d} - x^n$ with $d,\, n \ge 2.$ The point
$o=(0,0)$ is the only singularity of the affine curve $C_{d,
\,n}=V(F_{d, \, n})$, and the corresponding Milnor number is given
by
\begin{equation*}
\mu_{d,\,n}:=\mu(C_{d, \, n}, \, o)= \dim_{\mathbb{C}}
\frac{\mathbb{C}\{x, \, y\}}{(x^{n-1}, \, y^{d-1})}=(d-1)(n-1).
\end{equation*}
\end{example}

\section{Proofs of Theorems A and B} \label{sec:proofs}
We start by proving Theorem A.
\begin{teoA}
For every $d \geq 3$ there exists at least one proper polynomial map
$f \colon \mC^2 \lr \mC^2$ of topological degree $d$ such that $f$
is not equivalent to a map of type $(x, \, y) \lr (x, \, Q(x,y))$.
\end{teoA}
\begin{proof}
Consider the polynomial map $f_d \colon \mathbb{C}^2 \lr
\mathbb{C}^2$ defined as follows:
\begin{equation*}
f_d(x, y):=(x+y+xy, \, x^{d-1} y).
\end{equation*}
\begin{claim}
The map $f_d$ is proper and has topological degree $d$ for all $d
\geq 2$.
\end{claim}
Indeed, look at the push-forward map
\begin{equation*}
\begin{split}
f_{d \,\ast} \colon \mathbb{C}[s, \, t] & \lr \mathbb{C}[x, \, y] \\
& s \lr x+y+xy \\
& t \lr x^{d-1}y.
\end{split}
\end{equation*}
The element $x \in \mathbb{C}[x, \,y]$ satisfies the monic equation
of degree $d$
\begin{equation*}
X^d -sX^{d-1}+tX+t=0.
\end{equation*}
Analogously, the element $y$ satisfies the monic equation
\begin{equation*}
Y(s-Y)^{d-1} -t(1+Y)^{d-1}=0.
\end{equation*}
This shows that $f_{d \, \ast}\mathbb{C}[s, \, t] \subset
\mathbb{C}[x,\, y]$ is a integral extension of rings of degree $d$,
hence Proposition \ref{prop:Jel} implies that $f_d$ is a proper map
 of degree $d$. This proves our claim. \\ \\
Now we want to show that $f_d$ is not equivalent to any
semi-separate map
for all $d \geq 3$. \\
The Jacobian $J_{f_d}(x, \, y)$ splits as
\begin{equation}
J_{f_d}(x, \, y)=H_1(x, \, y)^{d-2}H_2(x, \, y),
\end{equation}
where
\begin{equation*}
H_1(x, \, y)=x, \quad H_2(x, \, y)= (2-d)xy+x-(d-1)y.
\end{equation*}
For all $d \geq 3$, the conic $V(H_2)$ is biholomorphic to
$\mathbb{C}^{\ast}$. Since $\mathbb{C}$ and $\mathbb{C}^{\ast}$ are
obviously not biholomorphic, it follows by Proposition
\ref{crit-bihol} and Lemma \ref{divides} that there exists no
semi-separate map equivalent to $f_d$, for all $d \geq 3$. This
concludes the proof of Theorem A.
\end{proof}

\begin{remark}
For $d=2$ the map $f_d$ is equivalent to a semi-separate one. Indeed,
 consider  $\Phi_1 , \, \Phi_2 \in \emph{Aut}(\mathbb{C}^2)$ defined by
\begin{equation*}
\Phi_1(x, \, y)=\bigg(\frac{x+y}{2}, \, \frac{x-y}{2} \bigg), \quad
\Phi_2(x, \, y)=(x^2+2x-y,  \, x^2-y).
\end{equation*}
\end{remark}
Then we have
\begin{equation*}
 f_2(x, \, y)=\Phi_2 \circ \tilde{f} \circ \Phi_1(x,y),
\end{equation*}
where $\tilde{f}(x, \, y)=(x, \, y^2)$, in accordance
with Theorem \ref{teo:Lamy}. \\ \\
Now let us prove Theorem B.
\begin{teoB}
For all positive integers $d, \, n$, with $d \geq 3$ and $n \geq 2$,
consider the polynomial map $f_{d, \,n} \colon \mathbb{C}^2 \lr
\mathbb{C}^2$ given by
\begin{equation*}
f_{d, \, n}(x,\,y):=(x, \, y^d-dx^ny).
\end{equation*}
Then $f_{d, \,n}$ and $f_{d, \, m}$ are equivalent if and only if
$m=n$. It follows that there exist infinitely many different
equivalence classes of proper polynomial maps $\mC^2 \lr \mC^2$ of
topological degree $d$.
\end{teoB}

\begin{proof}
The critical locus of the map $f_{d, \,n}$ is the affine curve
$C_{d-1, \, n}$ of equation $y^{d-1}-x^n=0$, whose unique
singularity is $o=(0, \, 0)$. The Milnor number of $C_{d-1, \, n}$
in $o$ is given by
\begin{equation*}
\mu_{d-1, \, n}=\mu(C_{d-1,\, n}, \,o)=(d-2)(n-1),
\end{equation*}
see Example \ref{ex:milnor}. Hence $\mu_{d-1,\, n}=\mu_{d-1, \,m}$
if and only if $m=n$. It follows by Theorem \ref{milnor} that the
curves $C_{d-1, \,n}$ and $C_{d-1, \,m}$ are not biholomorphic
 if $m \neq n$, therefore Proposition \ref{crit-bihol} implies that
 $f_{d, \, n}$ and $f_{d, \, m}$
are not equivalent if $m \neq n$. \\
This concludes the proof of Theorem B.
\end{proof}

\section{The case of Galois coverings} \label{sec:Galois}

Let $f \colon \mathbb{C}^2 \lr \mathbb{C}^2$ be a polynomial map
which is a Galois covering with finite Galois group $G$. By
Proposition \ref{prop:Jel},  $f$ is proper and its topological
degree equals $|G|$; moreover $G \subset
\textrm{Aut}(\mathbb{C}^2)$, and $f$ can be identified with the
quotient map $\mathbb{C}^2 \lr \mathbb{C}^2/G$. Since $G$ is a
finite group, we may assume $G \subset \textrm{GL}(2,\mathbb{C})$
by a polynomial change of coordinates (\cite[Corollary 4.4]{Ka79})
and, since $\mathbb{C}^2/G \cong \mathbb{C}^2$, it follows that
$G$ is a \emph{finite complex reflection group}. Let us denote by
$\mathbb{C}[x,\,y]^{G}$ the subalgebra of $G$-invariant
polynomials; then the following two conditions are equivalent, see
\cite[p.380]{Coh76}:
\begin{itemize}
\item[$(i)$]there are two algebraically independent homogenous
polynomials $\phi_1, \, \phi_2 \in \mathbb{C}[x,\, y]^{G}$ which
satisfy $|G|=\textrm{deg}(\phi_1) \cdot \textrm{deg}(\phi_2)$;
\item[$(ii)$] there are two algebraically independent homogeneous
polynomials $\phi_1, \, \phi_2 \in \mathbb{C}[x, \, y]^{G}$ such
that $1$, $\phi_1$, $\phi_2$  generate
 $\mathbb{C}[x,\,y]^{G}$ as an algebra over $\mathbb{C}$.
\end{itemize}
We say that $\phi_1, \, \phi_2$ are a \emph{basic set of invariants}
for $G$. Furthermore, putting $d_1:=\textrm{deg}(\phi_1)$,
$d_2:=\textrm{deg}(\phi_2)$, the set $\{d_1, \, d_2\}$ is
independent of the particular choice of $\phi_1,\, \phi_2$. We call
$d_1$, $d_2$ the \emph{degrees} of $G$.

\begin{proposition} \label{basic set equivalent}
Let $\phi_1, \, \phi_2$ and $\psi_1, \, \psi_2$ be two basic sets of
invariants for $G$. Then the two polynomial maps $\phi, \, \psi \
 \colon \mathbb{C}^2 \lr \mathbb{C}^2$ defined by
\begin{equation*}
\begin{split}
\phi(x,\, y)&=(\phi_1(x, \,y), \, \phi_2(x, \, y)),\\
\psi(x,\,y)&=(\psi_1(x, \,y), \, \psi_2(x, \, y))
\end{split}
\end{equation*}
are equivalent.
\end{proposition}
\begin{proof}
Set
\begin{equation*}
d_1=\deg(\phi_1)=\deg(\psi_1), \quad
d_2=\deg(\phi_2)=\deg(\psi_2),
\end{equation*}
with $d_1 \leq d_2$. Since both $\{1, \, \phi_1, \phi_2 \}$ and
$\{1, \psi_1, \psi_2 \}$ generate $\mathbb{C}[x, \, y]^{G}$, we may
express both $\phi_1$ and $\phi_2$ as polynomials in $\psi_1,
\psi_2$. Looking at the degrees, one sees that there are three cases. \\
$\bullet$ If $d_1 \nmid d_2$, then there exist $a, \, b \in
\mathbb{C}^*$ such that
\begin{equation*}
\phi_1=a \psi_1, \quad \phi_2= b \psi_2.
\end{equation*}
Set $\Phi(x, \, y)=(ax, \, by)$. \\
$\bullet$ If $d_1 | d_2$ and $d_1 \neq d_2$, set $s=d_2/d_1$. Then
there exist $a, \, c, \, d \in \mathbb{C},$ $ad \neq 0,$ such that
\begin{equation*}
\phi_1=a \psi_1, \quad \phi_2= c\psi_1^s+d \psi_2.
\end{equation*}
Set $\Phi(x, \, y)=(ax, \, cx^s+dy)$. \\
$\bullet$ If $d_1=d_2$, then there exist $a, \, b, \, c, \, d \in
\mathbb{C},$ $(ad-bc) \neq 0,$ such that
\begin{equation*}
\phi_1=a \psi_1 + b \psi_2, \quad \phi_2= c\psi_1+d \psi_2.
\end{equation*}
Set $\Phi(x, \, y)=(ax+by, \, cx+dy)$. \\ \\
In all cases $\Phi \in
\textrm{Aut}(\mathbb{C}^2)$ and $\phi=\Phi \circ \psi$. This
completes the proof.
\end{proof}
\begin{corollary} \label{cor:basic set}
Let $f \colon \mathbb{C}^2 \lr \mathbb{C}^2$ be a Galois covering with
finite Galois group $G$. Then $f$ is equivalent to the map $\phi(x,
y)=(\phi_1(x, \,y), \, \phi_2(x, \, y))$, where $\phi_1, \,\phi_2$
is any basic set of invariants for $G$.
\end{corollary}
It is well known that there exists a unitary inner product on
$\mathbb{C}^2$ invariant under $G$, hence we may assume that $G$ is
a subgroup of the unitary group $\textrm{U}(2)$, see \cite[p.
382]{Coh76}. There are two cases, according whether the
representation $G \subset \textrm{U}(2)$ is reducible or not.

\subsection{The reducible case} \label{sec:reducible}

 Assume that there exists a
$1$-dimensional linear subspace $V \subset \mathbb{C}^2$ which is
invariant under $G$; then its orthogonal complement $V^{\perp}$ is
also invariant (\cite[Chapitre 1]{Se71}), and up to a linear change
of coordinates we may assume $V= \langle e_1 \rangle$,
$V^{\perp}=\langle e_2 \rangle$, where $\{e_1, \, e_2 \}$ is the
canonical basis of $\mathbb{C}^2$. This means that $G$ is generated
by
\begin{equation*}
g_1(x, \,y)=(\theta_m x, \, y), \quad g_2(x, \, y)=(x, \, \theta_n
y),
\end{equation*}
where $\theta_m$ is a primitive $m$-th  root of unity and $\theta_n$
is a primitive $n$-th root of unity, respectively. Therefore we
obtain the following
\begin{proposition} \label{reducible}
Let $G \subset \emph{U}(2)$ be a reducible finite complex reflection
group acting on $\mathbb{C}^2$. Then, up to a change of coordinates,
we are in one of the following cases:
\begin{itemize}
\item[$(1)$] $G=\mathbb{Z}_m$, generated by $g= \left(
\begin{array}{cc}
1 & 0 \\
0 & \exp(2 \pi i/m)
\end{array}
\right)$;
\item[$(2)$] $G=\mathbb{Z}_m \times \mathbb{Z}_n$, generated by \\
$g_1= \left(
\begin{array}{cc}
\exp(2 \pi i/m) & 0 \\
0 & 1
\end{array}
\right)$ \; and \; $g_2= \left(
\begin{array}{cc}
1 & 0 \\
0 & \exp(2 \pi i/n)
\end{array}
\right)$.
\end{itemize}
\end{proposition}

\subsection{The irreducible case} \label{sec:irreducible}

The finite irreducible complex reflection groups were classified by
Shephard and Todd in \cite{ST54}. They found an infinite family
$G(m, \, p, \, 2)$, depending on two positive integer parameters
$m$, $p$, with $p | m$, and $19$ exceptional cases, that they
numbered from $4$ to $22$. We start by describing the groups
belonging to the infinite family. One has
\begin{equation*}
G(m, \, p, \, 2)=\mathbb{Z}_2 \ltimes A(m, \; p, \; 2),
\end{equation*}
where $A(m, \, p, \, 2)$ is the abelian group of order $m^2/p$ whose
elements are the matrices $\left(
                                  \begin{array}{cc}
                                    \theta^{\alpha_1} & 0 \\
                                    0 & \theta^{\alpha_2} \\
                                  \end{array}
                                \right)$,
with $\theta=\exp(2 \pi i/m)$ and $\alpha_1+\alpha_2 \equiv 0$ (mod
$p$), whereas $\mathbb{Z}_2$ is generated by $\left(
                                  \begin{array}{cc}
                                    0 & 1 \\
                                    1 & 0 \\
                                  \end{array}
                                \right)$.
In particular, $G(m, \, m, \, 2)$ is the dihedral group of order
$2m$.

\begin{proposition} \label{infinite family}
$(1)$ $G(m, \, p, \, 2)$ acts irreducibly on $\mathbb{C}^2$, except
in the case $G(2,\,2,\,2)$. In particular, $G(m, \, p, \, 2)$ is non-abelian
provided that $(m, \, p) \neq (2, \, 2)$. \\
$(2)$ The only groups in the family $G(m, \, p, \, 2)$ which are
isomorphic as abstract groups are $G(2, \, 1, \, 2)$ and $G(4, \, 4,
\, 2)$.
\end{proposition}
\begin{proof}
$(1)$ Suppose that $G=G(m, \, p, \, 2)$ leaves invariant a
nontrivial proper linear subspace $V \subset \mathbb{C}^2$. In particular,
$V$ must be invariant under the linear transformation $(x, \,y) \lr
(y, \,x)$, hence we may assume, up to an interchanging of $V$ and
$V^{\perp}$, that $V$ is the line $x-y=0$. As $A(m, \; p, \; 2)$
stabilizes $V$, all diagonal coefficients of an element of $A(m, \;
p, \; 2)$  must be equal. From this, one easily deduces that $m=p=2$.
On the other hand, it is obvious that $G(2,\;2,\;2) \cong
\mathbb{Z}_2 \times \mathbb{Z}_2$ acts reducibly on $\mathbb{C}^2$.
\\ \\
$(2)$ Assume that $G(m, \, p,\, 2)$ and $G(m', \, p', \, 2)$ are
isomorphic as abstract groups. In particular $|G(m, \, p, \,
2)|=|G(m', \, p', \, 2)|$ and $|Z(G(m, \, p, \, 2))|=|Z(G(m', \, p',
\, 2))|$. Setting $q=m/p$, $q'=m'/p'$, by \cite[p. 387]{Coh76} we
obtain
\begin{equation*}
mq=m'q', \quad q \cdot \textrm{gcd}(p,2)=q' \cdot
\textrm{gcd}(p',2).
\end{equation*}
If $\textrm{gcd}(p,2)=\textrm{gcd}(p',2)$ we have $q=q'$, hence
$m=m'$ and $p=p'$. Therefore we may suppose that $p$ is odd and $p'$
is even. Hence $q=2q'$, that is $m'=2m$ and $p'=4p$. Since $p'|m'$,
it follows that $m$ must be even. Summing up, we are left to
understand when $G(m, \, p, \, 2)$ and $G(2m, \, 4p, \, 2)$, $m$
even, $p$ odd are isomorphic as abstract groups. If $m$ is even and
$p$ is odd, there are exactly $m+3$ elements of order $2$ in $G(m,
\, p, \, 2)$, namely
\begin{equation*}
\left(
  \begin{array}{cc}
    \theta ^{m/2 }& 0 \\
    0 & 1 \\
  \end{array}
\right), \,
\left(
  \begin{array}{cc}
    1 & 0 \\
    0 & \theta^{m/2} \\
  \end{array}
\right), \,
\left(
  \begin{array}{cc}
    \theta ^{m/2 }& 0 \\
    0 & \theta^{m/2} \\
  \end{array}
\right) \, \textrm{and} \, \left(
  \begin{array}{cc}
    0 & \theta^{\alpha_1} \\
    \theta^{\alpha_2} & 0 \\
  \end{array}
\right),
\end{equation*}
where $\alpha_1+\alpha_2=m$. On the other hand, the two matrices
\begin{equation*}
\left(
  \begin{array}{cc}
    \theta ^m & 0 \\
    0 & 1 \\
  \end{array}
\right), \,
\left(
  \begin{array}{cc}
    1 & 0 \\
    0 & \theta^m \\
  \end{array}
\right)
\end{equation*}
belong to $G(2m,\, 4p,\, 2)$ if and only if $4|m$. So $G(2m, \, 4p,
\, 2)$ contains $2m+3$ elements of order $2$ if $4 |m$, and $2m+1$
elements of order $2$ if $4 \nmid m$. Consequently, if  $G(m, \, p,
\, 2)$ and $G(2m, \, 4p, \, 2)$, $m$ even, $p$ odd are isomorphic as
abstract groups the only possibility is $4 \nmid m$ and $m+3=2m+1$,
that is
 $m=2$, $p=1$. Finally, it is no difficult to check that $G(2, \, 1,
  \, 2)$ and $G(4, \, 4, \, 2)$ are conjugate in $\textrm{U}(2)$, hence they
  are isomorphic not only as abstract groups, but actually as
  complex reflection groups, see \cite[p. 388]{Coh76}.
\end{proof}
Now let us consider the exceptional groups in the Shephard-Todd's
list. We closely follow the treatment given in \cite{ST54}. For
$p=3, \, 4, \, 5$, the abstract group
\begin{equation*}
\langle s, \, t \, | \,s^2=t^3=(st)^p=1 \rangle
\end{equation*}
is isomorphic to $\mathcal{A}_4$, $\mathcal{S}_4$ and
$\mathcal{A}_5$, respectively. These are the well-known groups of
symmetries of regular polyhedra: $\mathcal{A}_4$ is the symmetry
group of the tetrahedron, $\mathcal{S}_4$ is the symmetry group of
the cube (and of the octahedron) and $\mathcal{A}_5$ is the symmetry
group of the dodecahedron (and the icosahedron). We take Klein's
representation of these groups by complex matrices (\cite{Kl84}),
and we call $S_1$, $T_1$ the matrices corresponding to the
generators $s$ and $t$, respectively. Therefore the exceptional
finite complex reflection groups are generated by matrices
\begin{equation*}
S=\lambda S_1, \quad T=\mu T_1, \quad Z= \exp(2 \pi i /k) I,
\end{equation*}
where $\lambda$, $\mu$ are suitably chosen roots of unity and $k$ is
a suitable integer. The corresponding abstract presentations are of
the form
\begin{equation} \label{group-presentation}
\langle S, \, T, \, Z \,| \, S^2=Z^{k_1}, \, T^3=Z^{k_2}, \,
(ST)^p=Z^{k_3}, \, [S,Z]=1, \, [T, Z]=1, \, Z^k=1 \rangle
\end{equation}
where $p=1, \, 2, \, 3$ and $k_1$, $k_2$, $k_3$, $k$ are suitably
chosen integers. We shall arrange the possible values of $\lambda$,
$\mu$, $k_1$, $k_2$, $k_3$, $k$ in tabular form, according to
Shephard-Todd's list (\cite[p. 280-286]{ST54}).  \\ \\
\emph{Exceptional groups derived from} $\mathcal{A}_4$. Set
$\omega=\exp(2 \pi i /3)$, $\varepsilon=\exp(2 \pi i /8)$. We have
\begin{equation*}
S_1=\left(
       \begin{array}{cc}
         i & 0 \\
         0 & -i \\
       \end{array}
     \right), \quad
T_1= \frac{1}{\sqrt{2}} \left(
       \begin{array}{cc}
         \varepsilon & \varepsilon^3 \\
         \varepsilon & \varepsilon^7 \\
       \end{array}
     \right).
\end{equation*}
The four corresponding groups are shown in Table
\ref{exceptional-from-A4} below. Here \verb|IdSmallGroup|$(G)$
denotes the label of $G$ in the \verb|GAP4| database of small
groups, which includes all groups of order less than $2000$, with the exception
of $1024$ (\cite{GAP4}). For instance, one has
\verb|[24,3]|$=\textrm{SL}_2(\mathbb{F}_3)$ and this means that
$\textrm{SL}_2(\mathbb{F}_3)$ is the third in the list of groups of
order $24$ (see the \verb|GAP4| script 1 in the Appendix).

\begin{table}[H]
\begin{center}
\begin{tabular}{c c c c c c c c c}
\hline
$ $ & \verb|IdSmall| & $ $ & $ $ & $ $ & $ $ & $ $ & $ $ & $ $ \\
No. & \verb|Group|$(G)$ & $\lambda$ & $\mu$ & $k_1$ & $k_2$ & $k_3$
& $k$ & Degrees \\
\hline
$4$ & \verb|[24,3]| & $-1$ & $-\omega$ & $1$ & $2$ & $2$ & $2$ & $4,\, 6$ \\
$5$ & \verb|[72,25]| & $- \omega$ & $- \omega$ & $1$ & $6$ & $6$ &
$6$ &
$6, \, 12$ \\
$6$ & \verb|[48,33]| & $i$ & $- \omega$ & $4$ & $4$ & $1$ & $4$ &
$4, \, 12$ \\
$7$ & \verb|[144,157]| & $i \omega$ & $- \omega$ & $8$ & $12$ & $3$
& $12$ & $12, \, 12$ \\ \hline
\end{tabular}
\end{center}
\caption{} \label{exceptional-from-A4}
\end{table}

\noindent \emph{Exceptional groups derived from} $\mathcal{S}_4$. We
have
\begin{equation*}
S_1= \frac{1}{\sqrt{2}}\left(
       \begin{array}{cc}
         i & 1 \\
         -1 & -i \\
       \end{array}
     \right), \quad
T_1= \frac{1}{\sqrt{2}} \left(
       \begin{array}{cc}
         \varepsilon & \varepsilon \\
         \varepsilon^3 & \varepsilon^7 \\
       \end{array}
     \right).
\end{equation*}
The eight corresponding groups are shown in Table
\ref{exceptional-from-S4} below.
\begin{table}[H]
\begin{center}
\begin{tabular}{c c c c c c c c c}
\hline
$ $ & \verb|IdSmall| & $ $ & $ $ & $ $ & $ $ & $ $ & $ $ & $ $ \\
No. & \verb|Group|$(G)$ & $\lambda$ & $\mu$ & $k_1$ & $k_2$ & $k_3$
& $k$ & Degrees \\
\hline
$8$ & \verb|[96,67]| & $\varepsilon^3$ & $1$ & $1$ & $2$ & $4$ & $4$ & $8,\, 12$ \\
$9$ & \verb|[192,963]| & $i$ & $\varepsilon$ & $8$ & $7$ & $8$ & $8$ & $8,\, 24$ \\
$10$ & \verb|[288,400]| & $\varepsilon^7 \omega^2$ & $- \omega$ &
$7$ & $12$ & $12$ & $12$ & $12,\, 24$ \\
$11$ & \verb|[576,5472]| & $i$ & $ \varepsilon \omega$ &
$24$ & $21$ & $8$ & $24$ & $24,\, 24$ \\
$12$ & \verb|[48,29]| & $i$ & $1$ &
$2$ & $1$ & $1$ & $2$ & $6,\, 8$ \\
$13$ & \verb|[96,192]| & $i$ & $i$ &
$4$ & $1$ & $2$ & $4$ & $8,\, 12$ \\
$14$ & \verb|[144,122]| & $i$ & $- \omega$ &
$6$ & $6$ & $5$ & $6$ & $6,\, 24$ \\
$15$ & \verb|[288,903]| & $i$ & $i \omega$ &
$12$ & $3$ & $10$ & $12$ & $12,\, 24$ \\
 \hline
\end{tabular}
\end{center}
\caption{} \label{exceptional-from-S4}
\end{table}

\noindent \emph{Exceptional groups derived from} $\mathcal{A}_5$.
Set $\eta= \exp(2 \pi i /5)$. We have
\begin{equation*}
S_1= \frac{1}{\sqrt{5}}\left(
       \begin{array}{cc}
         \eta^4 - \eta & \eta^2 - \eta^3 \\
         \eta^2 - \eta^3 & \eta - \eta^4 \\
       \end{array}
     \right), \quad
T_1= \frac{1}{\sqrt{5}} \left(
       \begin{array}{cc}
         \eta^2 - \eta^4 & \eta^4 -1 \\
         1 - \eta & \eta^3 - \eta \\
       \end{array}
     \right).
\end{equation*}
The seven corresponding groups are shown in Table
\ref{exceptional-from-A5} below.

\begin{table}[H]
\begin{center}
\begin{tabular}{c c c c c c c c c}
\hline
$ $ & \verb|IdSmall| & $ $ & $ $ & $ $ & $ $ & $ $ & $ $ & $ $ \\
No. & \verb|Group|$(G)$ & $\lambda$ & $\mu$ & $k_1$ & $k_2$ & $k_3$
& $k$ & Degrees \\
\hline
$16$ & \verb|[600,54]| & $-\eta^3$ & $1$ & $7$ & $10$ & $10$ & $10$ & $20,\, 30$ \\
$17$ & \verb|[1200,483]| & $i$ & $i \eta^3$ & $20$ & $11$ & $20$ & $20$ & $20,\, 60$ \\
$18$ & \verb|[1800,328]| & $- \omega \eta^3$ & $\omega^2$ &
$11$ & $30$ & $30$ & $30$ & $30,\, 60$ \\
$19$ & \verb|[3600, ]| & $i \omega$ & $i \eta^3$ &
$40$ & $33$ & $40$ & $60$ & $60,\, 60$ \\
$20$ & \verb|[360,51]| & $1$ & $\omega^2$ &
$3$ & $6$ & $5$ & $6$ & $12,\, 30$ \\
$21$ & \verb|[720,420]| & $i$ & $\omega^2$ &
$12$ & $12$ & $1$ & $12$ & $12,\, 60$ \\
$22$ & \verb|[240, 93]| & $i$ & $1$ &
$4$ & $4$ & $3$ & $4$ & $12,\, 20$ \\
\hline
\end{tabular}
\end{center}
\caption{} \label{exceptional-from-A5}
\end{table}

\begin{proposition} \label{prop:no-isomorphism}
None of the groups in Tables \emph{\ref{exceptional-from-A4}},
\emph{\ref{exceptional-from-S4}}, \emph{\ref{exceptional-from-A5}}
is isomorphic as an abstract group to some $G(m, \,p,\, 2)$.
\end{proposition}
\begin{proof}
Let $G$ be one of the groups in the tables. Looking at the
presentation \eqref{group-presentation}, one easily sees that the
center of $G$ is $\langle Z \rangle \cong \mathbb{Z}_k$ and that
this is the maximal normal abelian subgroup of $G$. Since in every
case $2k < |G|$, this implies that $G$ contains no normal abelian
subgroups of index $2$, hence it cannot be isomorphic to $G(m, \, p,
\, 2) = \mathbb{Z}_2 \ltimes A(m, \, p, \, 2)$.
\end{proof}

\begin{definition}
A finite group $G$ of unitary automorphisms of $\mathbb{C}^2$ is
called \emph{imprimitive} if $\mathbb{C}^2=V_1 \oplus V_2$, where
$V_1$ and
 $V_2$ are nontrivial proper linear subspaces such that the set
 $\{V_1, \, V_2 \}$ is invariant under $G$. If such a direct
 splitting of $\mathbb{C}^2$ does not exist, $G$ is called
 \emph{primitive}.
\end{definition}
Notice that every group $G(m, \, p, \, 2)$ is imprimitive, since we
can take $V_1= \langle e_1 \rangle$, $V_2 = \langle e_2 \rangle$. By
Proposition \ref{infinite family} and \cite[p. 386 and p.
 394]{Coh76} one obtains

\begin{proposition} \label{conjugacy-irreducible}
Let $G$ be an irreducible finite complex reflection group  acting
on $\mathbb{C}^2$. \\
$(1)$ If $G$ is imprimitive, then $G$ is conjugate in $\emph{U}(2)$
to $G(m,\, p,\, 2)$ for some $m, \, p \in \mathbb{N}$, $p|m$, $(m,
\,p) \neq (2,2)$. The pair $(m,\,p)$ is uniquely determined, with
the exception of $G(2,\,1,\,2)$ which is conjugate to $G(4,4,2)$.
\\
$(2)$ If $G$ is primitive, then $G$ is conjugate in $\emph{U}(2)$ to
exactly one of the groups $4, \ldots, 22$ in Tables $1$, $2$, $3$.
\end{proposition}

\subsection{The classification} \label{sub:classification}

Now we can give the classification, up to equivalence, of finite
Galois coverings $f \colon \mathbb{C}^2 \lr \mathbb{C}^2$. Set
\begin{equation*}
\begin{split}
\textsf{a}_4(x,\,y)&=x^4+(4 \xi -2)x^2y^2+y^4, \quad \xi=\exp(2 \pi i/6), \\
\textsf{b}_6(x,\, y)&=x^5y-xy^5, \\
\textsf{c}_8(x,\,y)&=x^8+14x^4y^4+y^8, \\ \textsf{d}_{12}(x, \,
y)&=
x^{12}-33x^8y^4-33x^4y^8+y^{12}, \\
\textsf{e}_{12}(x, \, y)&=x^{11}y+11x^6y^6-xy^{11}, \\
\textsf{f}_{20}(x, \,
y)&=x^{20}-228x^{15}y^5+494x^{10}y^{10}+228x^5y^{15}+y^{20}, \\
\textsf{g}_{30}(x,y)&=x^{30}+522x^{25}y^5-10005x^{20}y^{10}-10005x^{10}y^{20}-522x^5y^{25}+y^{30}.
\end{split}
\end{equation*}
Then we have

\begin{theorem} \label{teo:Galois}
Let $f \colon \mathbb{C}^2 \lr \mathbb{C}^2$ be a polynomial map
which is a Galois covering with finite Galois group $G$. Then, up to
equivalence, we are in one of the cases in Table
\emph{\ref{table:Galois}} below. Furthermore, these maps are
pairwise non-equivalent, with the only exception of $\mathfrak{f}_{2, \, 1, \,
2}$ and $\mathfrak{f}_{4, \, 4, \, 2}$.
\end{theorem}
\begin{table}[H]
\begin{center}
\begin{tabular}{c c c c}
\hline
Map & $\phi_1, \, \phi_2$ & $G$ & Branch locus \\
\hline
$\mathfrak{f}_m$ & $x, \, y^m$ & $\mathbb{Z}_m$ & $y=0$ \\
$\mathfrak{f}_{m, \, n}$ & $x^m, \, y^n$ & $\mathbb{Z}_m \times \mathbb{Z}_n$
& $xy=0$ \\
$\mathfrak{f}_{m, \, p, \, 2}$ & $x^{m/p}y^{m/p},\, x^m + y^m$ & $G(m, \,
p,\,
2)$ &  $x(y^2-4x^p)=0 \quad \textrm{if} \; p \neq m$ \\
 & & & $\quad \quad y^2-4x^p=0 \quad \textrm{if} \; p=m$ \\
$\tilde{\mathfrak{f}}_4$ & $\textsf{a}_4, \, \textsf{b}_6$ & $G_4=$\verb|[24, 3]| & $x^3+(-24\xi+12)y^2=0$ \\
$\tilde{\mathfrak{f}}_5$ & $\textsf{b}_6, \, (\textsf{a}_4)^3$ & $G_5=$\verb|[72, 25]| & $y(x^2+ (\frac{1}{18\xi}-\frac{1}{36})y)=0$  \\
$\tilde{\mathfrak{f}}_6$ & $\textsf{a}_4, \, (\textsf{b}_6)^2$ & $G_6=$\verb|[48, 33]| &  $y(x^3+(-24\xi+12)y^2)=0$\\
$\tilde{\mathfrak{f}}_7$ & $(\textsf{b}_6)^2, \, (\textsf{a}_4)^3$ & $G_7=$ \verb|[144, 157]| & $xy(x+ (\frac{1}{18\xi}-\frac{1}{36})y)=0$\\
$\tilde{\mathfrak{f}}_8$ & $\textsf{c}_8, \, \textsf{d}_{12}$ & $G_8=$\verb|[96, 67]| & $y^2-x^3=0$\\
$\tilde{\mathfrak{f}}_9$ & $\textsf{c}_8, \, (\textsf{d}_{12})^2$ & $G_9=$ \verb|[192, 963]| & $y(y-x^3)=0$ \\
$\tilde{\mathfrak{f}}_{10}$ & $\textsf{d}_{12}, \, (\textsf{c}_8)^3$ & $G_{10}=$\verb|[288, 400]| & $y(y-x^2)$=0 \\
$\tilde{\mathfrak{f}}_{11}$ & $(\textsf{d}_{12})^2, \, (\textsf{c}_8)^3$ & $G_{11}=$\verb|[576, 5472]| & $xy(x-y)=0$\\
$\tilde{\mathfrak{f}}_{12}$ & $\textsf{b}_6, \, \textsf{c}_8$ & $G_{12}=$\verb|[48, 29]| & $y^3-108x^4=0$\\
$\tilde{\mathfrak{f}}_{13}$ & $\textsf{c}_8, \, (\textsf{b}_6)^2$ & $G_{13}=$\verb|[96, 192]| & $y(x^3-108y^2)$=0 \\
$\tilde{\mathfrak{f}}_{14}$ & $\textsf{b}_6, \, (\textsf{d}_{12})^2$ & $G_{14}=$\verb|[144, 122]| & $y(y+108x^4)$=0 \\
$\tilde{\mathfrak{f}}_{15}$ & $(\textsf{b}_6)^2, \, (\textsf{d}_{12})^2$ & $G_{15}=$ \verb|[288, 903]| & $xy(y+108x^2)=0$  \\
$\tilde{\mathfrak{f}}_{16}$ & $\textsf{f}_{20}, \, \textsf{g}_{30}$ & $G_{16}=$\verb|[600, 54]| &  $y^2-x^3=0$\\
$\tilde{\mathfrak{f}}_{17}$ & $\textsf{f}_{20}, \, (\textsf{g}_{30})^2$ & $G_{17}=$\verb|[1200, 483]| & $y(y-x^3)=0$\\
$\tilde{\mathfrak{f}}_{18}$ & $\textsf{g}_{30}, \, (\textsf{f}_{20})^3$ & $G_{18}=$\verb|[1800, 328]| & $y(y-x^2)=0$\\
$\tilde{\mathfrak{f}}_{19}$ & $(\textsf{g}_{30})^2, \, (\textsf{f}_{20})^3$ & $G_{19}=$ \verb|[3600, ]| & $xy(x-y)=0$ \\
$\tilde{\mathfrak{f}}_{20}$ & $\textsf{e}_{12}, \, \textsf{g}_{30}$ & $G_{20}=$ \verb|[360, 51]| & $y^2-1728 x^5=0$\\
$\tilde{\mathfrak{f}}_{21}$ & $\textsf{e}_{12}, \, (\textsf{g}_{30})^2$ & $G_{21}=$\verb|[720, 420]| & $y(y-1728x^5)=0$\\
$\tilde{\mathfrak{f}}_{22}$ & $\textsf{e}_{12}, \, \textsf{f}_{20}$ & $G_{22}=$\verb|[240, 93]| & $y^3+1728x^5=0$ \\

\hline

\end{tabular}
\end{center}
\caption{} \label{table:Galois}
\end{table}
\begin{proof}
By Propositions \ref{reducible} and \ref{conjugacy-irreducible}, $G$
 is conjugate in $\textrm{U}(2)$ to one of the groups in Table \ref{table:Galois}.
Moreover by Propositions \ref{infinite family} and
\ref{prop:no-isomorphism} these groups are pairwise not isomorphic,
with the unique exception of $G(2, \, 1, \, 2)$ and $G(4, \, 4, \,
2)$. Therefore, by Corollary \ref{cor:basic set} it is sufficient to
show that in every case $\phi_1, \,\phi_2$ form a basic set of
invariants for $G$. This is obvious in the first three cases. For
the remaining groups we can do a case-by-case analysis, using the
description of $G$ given in Subsections \ref{sec:reducible} and
\ref{sec:irreducible}. A shorter proof can be obtained by noticing
that: \\
$-$ $\textsf{a}_4$ is $G_4$-invariant and, up to a multiplicative
constant, $\textsf{b}_6=\textrm{Jacobian}(\textsf{a}_4, \,
\textrm{Hessian}(\textsf{a}_4))$; \\
$-$ $\textsf{b}_6$ is $G_{12}$-invariant and, up to multiplicative
constants, $\textsf{c}_8=\textrm{Hessian}(\textsf{b}_6)$ and
$\textsf{d}_{12}=\textrm{Jacobian}(\textsf{b}_6, \,
\textsf{c}_8)$;
\\
$-$ $\textsf{e}_{12}$ is $G_{20}$-invariant and, up to
multiplicative constants,
$\textsf{f}_{20}=\textrm{Hessian}(\textsf{e}_{12})$ and
$\textsf{g}_{30}=\textrm{Jacobian}(\textsf{e}_{12}, \,
\textsf{f}_{20})$.
\\
Then $\phi_1, \, \phi_2$ form a basic sets of invariants for $G_{4},
\ldots, G_{22}$ by \cite[p. 285-286]{ST54}, \cite{Chev55}, \cite{Kl84}. \\
Finally, the computation of the branch locus in each case is a
straightforward application of elimination theory and can be carried
out with the help of the Computer Algebra System \verb|Singular|
(\cite{SING}). Look at the \verb|Singular| script
$3$ in the Appendix to see how this applies to an explicit example,
namely the map $\tilde{\mathfrak{f}}_4$.
\end{proof}

The following corollary generalizes Theorem \ref{teo:Lamy}
to the case of Galois coverings of arbitrary degree.

\begin{corollary} \label{cor:finite-galois}
For all $d \geq 2$, there exist only finitely many equivalence
classes of Galois coverings $f \colon \mathbb{C}^2 \lr \mathbb{C}^2$
 of topological degree $d$.
\end{corollary}
\begin{proof}
For all $d \geq 2$, there are only finitely many integers $m$, $n$,
$p$ such that any of the equalities
\begin{equation*}
|\mathbb{Z}_{m}|=d, \quad |\mathbb{Z}_m \times \mathbb{Z}_n|=d,
\quad |G(m, \, p, \, 2)|=d
\end{equation*}
holds.
\end{proof}

\begin{remark}
The computation of the invariant polynomials $\textsf{a}_4, \ldots,
\textsf{g}_{30}$ goes back to Klein, see \emph{\cite{Kl84}}.
Nowadays, it can be easily carried out by using the \verb|Singular|
script $2$ in the Appendix.
\end{remark}

\begin{remark} \label{rem:already-studied}
Some of the coverings in Table \emph{\ref{table:Galois}} already
appeared in the literature. For instance, those with groups $G(m,\,
m,\, 2)$, $G_4$, $G_8$, $G_{16}$, $G_{20}$ were studied $($by
different methods$)$
 in \emph{\cite{NT95}}, whereas those with groups $G(m, \, 1, \, 2)$,
$G_5$, $G_6$, $G_9$, $G_{10}$, $G_{14}$, $G_{17}$, $G_{18}$,
$G_{21}$ were studied in \emph{\cite{NN00}}.
\end{remark}

\begin{remark}
The critical locus of every map in Table \emph{4} is a finite union
of lines through the origin, since the
two components $\phi_1$, $\phi_2$ are always homogeneous polynomials.
Therefore in each case the origin is a total ramification point and this in turn
implies that all these examples are ``polynomial-like" self-maps of
$\mathbb{C}^2$, see \emph{\cite[Example 2.1.1]{DS08}}. In
particular, their dynamical behaviour has been largely investigated,
see \emph{\cite{DS03}}.
\end{remark}

\section*{Appendix}

In this appendix we include, for the reader's convenience, some of
the \verb|GAP4| and \verb|Singular| scripts that we have used in our
computations; all the others are similar and can be easily obtained
modifying the ones below. \\
The \verb|GAP4| script $1$ finds the label \verb|[24, 3]| of the
group $G_4$ in Table \ref{table:Galois} and shows that it is
isomorphic to $\textrm{SL}_2(\mathbb{F}_3)$. The \verb|Singular|
script $2$ computes the basic set of invariants $\textsf{a}_4$,
$\textsf{b}_6$ for $G_4$, whereas the \verb|Singular| script $3$
shows that the branch locus of the map $\tilde{\mathfrak{f}}_4(x, \,
y)=(\textsf{a}_4(x,\, y), \, \textsf{b}_6(x,\,y))$ is the curve
$x^3+(-24 \exp(2 \pi i /6)+12)y^2=0$.

\bigskip

\begin{verbatim}
gap> ####### GAP4 SCRIPT 1: Identifying the groups #######
gap> F:=FreeGroup("s", "t", "z");; gap> s:=F.1;; t:=F.2;; z:=F.3;;
gap> #insert the presentation of G
gap> G:=F/[s^2*z^-1, t^3*z^-2, (s*t)^3*z^-2,
> z*s*z^-1*s^-1, z*t*z^-1*t^-1, z^2];;
gap> # compute the label of G
gap> IdSmallGroup(G);
[24, 3]
gap> # check that G is isomorphic to SL(2,3)
gap> G1:=SL(2,3);; IdSmallGroup(G1);
[24, 3]
\end{verbatim}

\bigskip

\begin{verbatim}
> ;  // ------- SINGULAR SCRIPT 2: Finding the invariants -------
> LIB("finvar.lib");
> ring R =(0,a), (x, y), dp;
> ;  // minimal polynomial of a=exp(2 pi i/24)
> minpoly = a^8-a^4+1;
> number e=a^3;
> number w=a^8;
> number i=e^2;
> number r2=e-e^3; // r2=sqrt(2)
> ;  // define the matrices S1 and T1
> matrix S1[2][2]= i, 0, 0, -i;
> matrix T1[2][2]= e*r2^-1, e^3*r2^-1, e*r2^-1, e^7*r2^-1;
> ;  // define the matrices S and T
> matrix S = -S1;
> matrix T = -w * T1;
> ;  // compute a basic set of invariants
> ;  // for the group generated by S and T
> invariant_ring(S, T);
_[1,1]=x4+(4a4-2)*x2y2+y4 _[1,2]=x5y-xy5 _[1,1]=1 _[1,1]=0
\end{verbatim}

\bigskip

\begin{verbatim}
> ;  // ------- SINGULAR SCRIPT 3: Computing the branch locus -------
> ring R = (0, a), (s, t, x, y), dp;
> ; // minimal polynomial of a=exp(2 pi i/6)
> minpoly = a^2-a+1;
> ;  // define the map f(s,t)=(X(s,t), Y(s,t))
> poly X = s4+(4a-2)*s2t2+t4;
> poly Y = s5t-st5;
> ;  // compute the Jacobian of f
> poly j = diff(X,s)*diff(Y,t)-diff(X,t)*diff(Y,s);
> ideal I = j, x-X, y-Y;
> ;  // compute the equation of the branch curve B(f)
> ;  // by eliminating the variables s, t
> ideal J = eliminate(I,  st);
> J;
J[1]=x3+(-24a+12)*y2

\end{verbatim}

\bigskip \bigskip
CINZIA BISI \\
Dipartimento di Matematica, Universit\`a della Calabria, Via P. Bucci
Cubo 30B, \\
87036 Arcavacata di Rende (CS), Italy. \\
\emph{E-mail address}: \verb|bisi@mat.unical.it| \\ \\

FRANCESCO POLIZZI \\
Dipartimento di Matematica, Universit\`a della Calabria, Via P. Bucci
Cubo 30B, \\
87036 Arcavacata di Rende (CS), Italy. \\
\emph{E-mail address}: \verb|polizzi@mat.unical.it|

\end{document}